\documentclass[12pt]{amsart}
\usepackage[colorlinks=true, pdfstartview=FitV, linkcolor=blue, citecolor=blue, urlcolor=blue]{hyperref}

\usepackage{amssymb,amsmath, amscd}
\usepackage{times, verbatim}
\usepackage{graphicx}
\usepackage[english]{babel}
 \usepackage[usenames, dvipsnames]{color}
\usepackage{amsmath,amssymb,amsfonts}
\usepackage{enumerate}
\usepackage{anysize}
\marginsize{3cm}{3cm}{3cm}{3cm}
\input xy
\xyoption{all}
\usepackage{pb-diagram}
\usepackage[all]{xy}
\input xy
\xyoption{all}

\DeclareFontFamily{OT1}{rsfs}{}
\DeclareFontShape{OT1}{rsfs}{n}{it}{<-> rsfs10}{}
\DeclareMathAlphabet{\mathscr}{OT1}{rsfs}{n}{it}

\DeclareFontFamily{U}{wncy}{}
    \DeclareFontShape{U}{wncy}{m}{n}{<->wncyr10}{}
    \DeclareSymbolFont{mcy}{U}{wncy}{m}{n}
    \DeclareMathSymbol{\Sh}{\mathord}{mcy}{"58}

\begin{document}
\theoremstyle{plain}

\newtheorem{theorem}{Theorem}[section]
\newtheorem{thm}[equation]{Theorem}
\newtheorem{prop}[equation]{Proposition}
\newtheorem{cor}[equation]{Corollary}
\newtheorem{conj}[equation]{Conjecture}
\newtheorem{lemma}[equation]{Lemma}
\newtheorem{definition}[equation]{Definition}
\newtheorem{question}[equation]{Question}
\theoremstyle{definition}
\newtheorem{remark}[equation]{Remark}
\newtheorem{example}[equation]{Example}
\numberwithin{equation}{section}

\newcommand{\Hecke}{\mathcal{H}}
\newcommand{\Liea}{\mathfrak{a}}
\newcommand{\Cmg}{C_{\mathrm{mg}}}
\newcommand{\Cinftyumg}{C^{\infty}_{\mathrm{umg}}}
\newcommand{\Cfd}{C_{\mathrm{fd}}}
\newcommand{\Cinftyfd}{C^{\infty}_{\mathrm{ufd}}}
\newcommand{\sspace}{\Gamma \backslash G}
\newcommand{\Sym}{{\rm Sym}}

\newcommand{\PP}{\mathcal{P}}
\newcommand{\bfP}{\mathbf{P}}
\newcommand{\bfQ}{\mathbf{Q}}
\newcommand{\Siegel}{\mathfrak{S}}
\newcommand{\g}{\mathfrak{g}}
\newcommand{\A}{{\mathbb A}}
\newcommand{\B}{{\rm B}}
\newcommand{\Q}{\mathbb{Q}}
\newcommand{\Gm}{\mathbb{G}_m}
\newcommand{\kk}{\mathfrak{k}}
\newcommand{\nn}{\mathfrak{n}}
\newcommand{\tF}{\tilde{F}}
\newcommand{\p}{\mathfrak{p}}
\newcommand{\m}{\mathfrak{m}}
\newcommand{\bb}{\mathfrak{b}}
\newcommand{\Ad}{{\rm Ad}\,}
\newcommand{\ttt}{\mathfrak{t}}
\newcommand{\frakt}{\mathfrak{t}}
\newcommand{\U}{\mathcal{U}}
\newcommand{\Z}{\mathbb{Z}}
\newcommand{\G}{\mathbb{G}}
\newcommand{\bfG}{\mathbf{G}}
\newcommand{\bfT}{\mathbf{T}}
\newcommand{\R}{\mathbb{R}}
\newcommand{\ST}{\mathbb{S}}
\newcommand{\h}{\mathfrak{h}}
\newcommand{\bC}{\mathbb{C}}
\newcommand{\C}{\mathbb{C}}
\newcommand{\E}{\mathbb{E}}
\newcommand{\F}{\mathbb{F}}
\newcommand{\N}{\mathbb{N}}
\newcommand{\qH}{\mathbb {H}}
\newcommand{\temp}{{\rm temp}}
\newcommand{\Hom}{{\rm Hom}}
\newcommand{\Aut}{{\rm Aut}}
\newcommand{\Ext}{{\rm Ext}}
\newcommand{\Nm}{{\rm Nm}}
\newcommand{\End}{{\rm End}\,}
\newcommand{\Ind}{{\rm Ind}\,}
\def\circG{{\,^\circ G}}
\def\M{{\rm M}}
\def\diag{{\rm diag}}
\def\Ad{{\rm Ad}}

\def\H{{\rm H}}
\def\SL{{\rm SL}}
\def\PSL{{\rm PSL}}
\def\GSp{{\rm GSp}}
\def\PGSp{{\rm PGSp}}
\def\Sp{{\rm Sp}}
\def\St{{\rm St}}
\def\GU{{\rm GU}}
\def\SU{{\rm SU}}
\def\U{{\rm U}}
\def\GO{{\rm GO}}
\def\GL{{\rm GL}}
\def\PGL{{\rm PGL}}
\def\GSO{{\rm GSO}}
\def\Gal{{\rm Gal}}
\def\SO{{\rm SO}}
\def\O{{\rm  O}}
\def\sym{{\rm sym}}
\def\St{{\rm St}}
\def\tr{{\rm tr\,}}
\def\ad{{\rm ad\, }}
\def\Ad{{\rm Ad\, }}
\def\rank{{\rm rank\,}}

    \title[Relating Tate-Shafarevich group with class group]
      {Relating Tate-Shafarevich group of an elliptic curve \\with class group}

      \author{Dipendra Prasad and Sudhanshu Shekhar}

\address{D.P.: Indian Institute of Technology Bombay, Powai, Mumbai-400076} 
\address{D.P.:  Tata Institute of Fundamental
Research, Colaba, Mumbai-400005.}
\email{prasad.dipendra@gmail.com}

\address{S.S.: Indian Institute of Technology, Kanpur.}
\email{sshekhars2012@gmail.com}
\maketitle
{\hfill \today}

\newcommand{\surj}{\twoheadrightarrow}
\newcommand{\inj}{\hookrightarrow}
\newcommand{\lrta}{\longrightarrow}
\newcommand{\WH}{\widehat}

\newcommand{\QQ}{\mathbb{Q}}

\newcommand{\ZZ}{\mathbb{Z}}
\newcommand{\FF}{\mathbb{F}}

\newcommand{\tor}{\mathrm{Tor}}
\newcommand{\Cl}{\mathrm{Cl}}
\newcommand{\Sel}{\mathrm{Sel}}
\newcommand{\im}{\mathrm{im}}
\newcommand{\Ker}{\mathrm{Ker}}
\newcommand{\ras}{\mathrm{res}}
\newcommand{\ET}{\mathrm{E}}
\newcommand{\rk}{\mathrm{rank}}
\newcommand{\sel}{\mathrm{Sel}}
\newcommand{\gal}{\mathrm{Gal}}
\newcommand{\cyc}{\mathrm{cyc}}
\newcommand{\coker}{\mathrm{Coker}}
\newcommand{\Img}{\mathrm{Im}}
\newcommand{\corank}{\mathrm{corank}}
\newcommand{\Aa}{\mathrm{A}}
\newcommand{\TT}{\bold{T}}

\newcommand{\crd}{\color{red}}

\newcommand{\ilim}{\varinjlim}
\newcommand{\plim}{\varprojlim}

\begin{abstract} The paper formulates a precise relationship between the Tate-Shafarevich group $\Sh(E)$ of an elliptic curve $E$ over $\Q$ with a quotient of the classgroup of $\Q(E[p])$ on which   $\Gal(\Q(E[p]/\Q) = \GL_2(\Z/p)$ operates by its standard
  2 dimensional representation over $\Z/p$. We establish such a relationship in most cases when $E$ has good reduction at $p$. 
\end{abstract}
\tableofcontents

\section{Introduction}
 Let $E$ be an elliptic curve over ${\Bbb Q}$, $p$ an odd prime number, 
and  
$\rho :  \Gal( \bar{\Q} / {\Q} )\longrightarrow  \GL_2 ({\Bbb F}_p)$ 
the associated
 Galois representation on elements of order $p$ on $E$. Assume that 
the image of the  Galois 
representation is all of  $\GL_2({\Bbb F}_p)$. The representation $\rho$ then gives 
rise to an extension $K$ of ${\Bbb Q}$ with Galois 
group $\GL_2({\Bbb F}_p)$. 
Let $\Cl_K$ denote the class group of $K$. The group $\GL_2({\Bbb F}_p)$ being the
Galois group of $K$ over ${\Bbb Q}$, operates on $\Cl_K$, hence on the 
${\Bbb F}_p$-vector space $\Cl_K/p\Cl_K$.
Write the {\it semi-simplification} of the representation 
of  $\GL_2({\Bbb F}_p)$ on $\Cl_K/p\Cl_K$ as $\sum V_\alpha$,  where 
$V_\alpha$'s are the various irreducible 
representations of $\GL_2({\Bbb F}_p)$ in characteristic $p$. It 
is a well-known fact that any irreducible
representation of $\GL_2({\Bbb F}_p)$ in characteristic $p$ is of the form 
$V_{i,j}=
{ \Sym}^i\otimes {\rm det}^j, \ 0\leq i \leq p-1, \ 0 \leq j\leq p-2$ where $\Sym^i$ refers to the 
$i$-th symmetric power of the standard 2 dimensional representation of
$\GL_2({\Bbb F}_p)$, and $\det$ denotes the determinant character 
of $\GL_2({\Bbb F}_p)$.

It is a natural question to understand which $V_{i,j}$'s appear in $\Cl_K/p\Cl_K$.
The aim of this note is to formulate some questions in this direction 
which can be viewed as a $\GL_2$ analogue of the famous theorem of Herbrand-Ribet, 
\cite{Ri1}, and \cite{Was} for an exposition. One is hoping for a 
conjectural answer along the lines of Herbrand-Ribet to say that the representation 
$\Sym^i(\F_p+\F_p) \otimes \det^j$ of $\GL_2(\F_p)$ appears in $\Cl_K/p\Cl_K$
if and only if the `algebraic part' of the first nonzero derivative of  $L(s, \Sym^i(E) \otimes \det^j)$ at 0 
is divisible by $p$ (in some favorable situations such as when $E$ has good reduction at $p$).

The authors have not seen any computation of the $\GL_2(\F_p)$ representation on $\Cl_K/p\Cl_K$ into a sum of irreducible 
pieces (after semi-simplification) where 
$K = \Q(E[p])$ is a Galois extension of $\Q$, say
 with $\Gal(K/\Q) = \GL_2({\Bbb F}_p)$. Presumably it is not beyond present computational powers to do 
such a computation say for $p=5$; this would be very useful data to have for the problems discussed in this paper.
There are 
examples known due to K. Rubin and A. Silverberg \cite{RS} of families of elliptic curves with the same $\Q(E[5])$ with Galois group
$\GL_2(\F_5)$. Are there  elliptic curves of the same rank, say = 0, in this family, for which
the $5$-valuation of the algebraic part of $L(E,1)$ as in the Birch-Swinnerton-Dyer  conjecture are different? If our suggestions
in this paper are correct then this should not happen!

In an unfinished manuscript of the first author \cite{P2}, a heuristic relating representations of $\GL_2(\F_p)$ on $H_K/pH_K$ and divisibility of certain $L$-values was given based on factorisation of the class number formula for the Dedekind zeta
function $\zeta_K(s)$: 
$$\zeta_K(s) = -\frac{hR}{w}s^{r_1+r_2-1}+ 
{\rm
  ~~higher~order~terms...},$$
in terms of the complex representation theory of $\GL_2(\F_p)$ on the left hand side of the class number formula, and in terms of mod $p$ representation theory on the right hand side of the classnumber formula involving $H_K/pH_K$. We will not detail the heuristic considerations made in \cite{P2} except to say a few words on congruences of automorphic representations and of their $L$-values.

We first
fix some notation. We will fix an isomorphism of $\bar{\Bbb Q}_p$ 
with $\C$ where $\bar{\Bbb Q}_p$ is
 a fixed algebraic closure of ${\Bbb Q}_p$, the field
of $p$-adic numbers. This allows one to define $\p$, a prime ideal in  $\bar{\Bbb Z}$,  the integral closure of
${\Bbb Z}$ in ${\Bbb C}$, over the prime ideal generated by $p$ in $\Z$.

One  defines congruence 
modulo a  prime $\mathfrak{p}$ in $\bar{\Bbb Q}$, 
of  automorphic representations of $G(\A)$ for $G$ a reductive 
algebraic group   over ${\Bbb Q}$  most simply through
congruence of 
Hecke eigenvalues (i.e., of the algebra ${\mathcal H}(G(\Z_p)\backslash G(\Q_p)/G(\Z_p), \Z)$) 
at places of ${\Bbb Q}$ where $G$ is unramified, 
i.e., 
is quasi-split and splits over an unramified extension; for the notion of congruences of automorphic
representations, it is best
to demand congruence of Hecke eigenvalues only at almost all places of ${\Bbb Q}$ where $G$ is unramified, 
omitting an unspecified finite set of places of ${\Bbb Q}$ including those where the group is ramified.
(To be able to talk of congruences presupposes algebraicity of Hecke eignevalues.) Observe that transferring representations on a group $G_1(\A)$ to $G_2(\A)$ via a map of the L-groups ${}^LG_1(\C) \rightarrow {}^LG_2(\C)$ which carries congruent automorphic
representations of $G_1(\A)$ to  congruent automorphic representations of $G_2(\A)$, allows one to talk
of congruences of automorphic representations on two very different groups by embedding their L-groups in a (suitable)
common L-group!

For $\GL_2(F)$, where $F$ is a totally real number field, 
congruences for holomorphic elliptic 
modular forms are usually defined using  Fourier expansions. For the symplectic similitude group $\GSp_{2n}(F)$, and the symplecic group $\Sp_{2n}(F)$, $F$ a totally real number field, there is still the notion of Fourier
expansion, and one can hope that at least for these groups, algebraicity (resp. $\p$-integrality) of Hecke eignevalues is equivalent to
algebraicity (resp. $\p$-integrality) of Fourier coefficients (up to a scaling) of an appropriate `newform', 
and that further, congruences too can be read off from Fourier expansion. The recent work of  Furusawa and Morimoto \cite{FM} building on 
earlier work of Dickson, Pitale, Saha and Schmidt on `B\"ocherer conjecture' for $\GSp_4(\A_\Q)$ 
relating Fourier coefficients to central $L$-values (which themselves are supposed to 
have rationality and congruence properties)
lends support to such an expectation, although many others have advocated a contrary expectation, see for example, end of page 251 of \cite{Ha}.  Modular forms of half integral weight for $\GL_2(F)$, where $F$ is a totally real number field, have very similar general features in that their Fourier coefficients too are related to central $L$-values (by the work of Waldspurger), presumably rationality and  integrality up to a scalar are the same as that of the Hecke-operators.

If there was a ($p$-adic) Galois representation associated to the automorphic
representations under consideration, one could call two automorphic representations congruent if the associated Galois
representations with values in $\GL_n(\overline{{\Bbb Q}}_p)$ 
were isomorphic 
modulo $p$ (after semi-simplification).

We next briefly talk about $L$-values. According to the conjectures of Bloch and Kato \cite{BK}, the highest nonzero derivative
of a motivic $L$-function at any integral point can be written 
canonically as a product of an algebraic part, and a transcendental part, allowing one to talk of 
primes of  $\overline{{\Bbb Q}}$ dividing an $L$-value, which would be
a short form for talking of  primes of  $\overline{{\Bbb Q}}$ dividing the
algebraic part of the  (highest nonzero derivatives of) the $L$-function at an integer.
We recall, as is standard in the subject, that when speaking
of $L$-functions of automorphic representations which are congruent modulo $\mathfrak{p}$, 
the Euler factors above $p$ are to be
removed.
In some cases, we could be lucky to be
dealing with an $L$-value which is itself algebraic. 
This is for instance the case 
for Herbrand-Ribet's theorem \cite{Ri1}, 
and \cite{P} for a more general context involving $L$-values of finite Artin representations cutting out CM number fields.
(Perhaps, highest nonzero derivative of  an irreducible $L$-function  at an integer point is never algebraic unless
one is dealing with these CM Artin representations!)
In this paper, we will be comparing an Artin L-function with one of an Elliptic curve which by the Birch-Swinnerton-Dyer
conjecture has a well-formulated algebraic part $L(1,E)^{\rm alg}$.

The  questions formulated in \cite{P2} 
are predicated on the following, general but somewhat vague
principle,  which seems to lie at the basis
of being able to define $p$-adic $L$-functions:
 if two automorphic representations  are congruent modulo $\mathfrak{p}$, then the algebraic
 parts of their $L_S$-values (where $L_S$ denotes partial $L$-function, removing Euler factors at $S$ which
 is any finite set of places of $\Q$  containing  $p, \infty$ and any prime where
 either of the two representations is ramified)  at any integral point  are both $\p$-integral,
 or neither are $\p$-integral, 
and if $p$-integral, they are congruent modulo $\mathfrak{p}$.

\section{Main question}
The following question from \cite{P2} is at the basis of this paper.

\begin{question}\label{ques}
Let $E$ be an elliptic curve over ${\Bbb Q}$ such that $E(\Q)=0$.
Let $K = \Q(E[p])$ be the Galois extension of $\Q$ obtained by attaching elements of order $p$ 
on $E$ where $p$ is an odd prime.
We assume that $\Gal(K/\Q) = \GL_2({\Bbb F}_p)$, 
and also that $p$ is coprime to $c_\ell = [E(\Q_\ell):E(\Q_\ell)^0]$, the so-called Tamagawa factors,  for all 
finite primes $\ell$.
Let $\Cl_K$ denote the class group of $K$ which comes equipped with a natural action of  $\Gal(K/\Q)= \GL_2({\Bbb F}_p)$.
Then if $p | |\Sh(E)(\Q)|$, is it true that 
the $\GL_2(\F_p)$ representation $\Cl_K/p\Cl_K$ 
contains  the standard
2-dimensional representation of $\GL_2(\F_p)$ as a quotient? What about the converse?
\end{question}

There are two natural appoaches to attack this question: one which is what we will discuss in greater
detail from next section, using Selmer groups, and the machinery of Galois cohomology available to deal with
it. The other approach,  as in the pioneering work of Ribet \cite{Ri1} is  by looking at the congruence of 
cusp forms  with an Eisenstein series. The method of Ribet has two basic steps:

\begin{enumerate}
\item finding 
  cusp forms on $\GL_2(\A_\Q)$ which are congruent to a given Eisenstein series if the constant term of the Eisenstein series
  is zero mod $p$,
\item proving that the associated mod $p$ representation of the cusp form serves the purpose of constructing everywhere
  unramified extension of $\Q(\mu_p)$ of the desired kind.
  \end{enumerate}

Neither of the two steps is understood for $\GSp(4)$ (perhaps not even for $\GL_2(\A_\Q)$ in some generality?), except that there is a recent work due to Bergstr\"om and Dummigan, \cite{BD} where they formulate some general questions along these lines.
Note that 
 there are two conjugacy classes
of maximal parabolic subgroups in $\GSp(4)$ and both could be sources of such congruences,
and therefore could have applications to Galois representations. We make very brief comment on what the method of Ribet
will give if the two steps above were accomplished.

We first take up the case when the Eisenstein series is one on Klingen parabolic associated to a cusp
form (on the Levi subgroup of the Klingen parabolic)
so that the associated Galois representation lands inside the dual parabolic which is the Siegel
parabolic. In this case, a congruence between Eisenstein series and a cusp form will be a source of
$\Sym^2(\F_p + \F_p) \otimes \det^j$ appearing in $\Cl_K/p\Cl_K$.

Next we consider an Eisenstein series supported on a Siegel parabolic (so that the associated Galois representation lands 
inside the Levi of the Klingen parabolic). The Klingen parabolic $P$  in
$\GSp(4)$ looks like 
$$ 1 \rightarrow N \rightarrow P \rightarrow \GL(2) \times {\bf G}_m 
\rightarrow 1,$$
with N a non-abelian unipotent group of dimension 3 which is 
the 3-dimensional Heisenberg group, and thus has a centre of dimension 1. 
Dividing $N$ by the center one gets 2 dimensional representation of $\GL(2)$
which thus seems ideally suited to give rise to extensions of
$K$ on which the Galois group of $K$, i.e., $\GL_2({\Bbb F}_p)$, operates by a 
twist of the 2 dimensional standard representation. 

The specific question thus
is that if $p$ divides (the algebraic part of) $L(1,\pi \otimes \omega^j)$, there is a cusp form on $\GSp(4)$ which is congruent to an 
Eisenstein series.  

For a specific example of such a congruence, see the paper of Harder \cite{Ha} where for the unique elliptic modular form $f$ for $\SL_2(\Z)$ of weight 22,  he conjectured existence (which was later confirmed)
of certain cuspidal eigenform on $\GSp(4,\A_\Q)$ with the following congruence for Hecke eigenvalues, 
which we just write down from his paper without further explanations:
$$\lambda(p) \equiv p^8 + a_f(p) + p^{13} \bmod 41, \quad {\rm ~~for~~ all~~ primes~~} p,$$
except to note that 
$\frac{\Lambda(f,14)}{\Omega_+}$ is an integer, and the crucial reason for the prime 41 is:
$$41 | \frac{\Lambda(f,14)}{\Omega_+}.$$

In the context of Harder's example,
an optimistic hope will be that the mod $41$ Galois representation associated to the eigenform $f$ of weight 22
for $\SL_2(\Z)$ cuts out a $\GL_2(\F_{41})$ extension $K_f$ of $\Q$ which has an unramified abelian extension
on which $\GL_2(\F_{41})$ acts by $ \Sym^1 \otimes \det ^{-8}$ (or, is it $ \Sym^1 \otimes \det ^{-13}$?).

\section{Class group and Tate-Shafarevich group}\label{intro}
Rest of this paper deals exclusively with constructing the standard 2-dimensional representation of $\GL_2(\F_p)$
on the class group $\Cl_K/p\Cl_K$ using Selmer group of $E$, an elliptic curve defined over $\Q$.

Let $\Sel_p(E/\Q)$ be the $p$-Selmer group of $E$ over $\Q$ defined by the exact sequence
\[ 0 \lrta \Sel_p(E/\Q) \lrta H^1(\Q,E[p]) \lrta \prod_{v}H^1(\Q_v,E(\overline{\Q}_v)), \]
where $v$ varies over primes of $\Q$. Since the restriction map $H^1(\Q,E[p]) \lrta H^1(\Q_v,E(\overline{\Q}_v))$ factors through $H^1(\Q,E[p]) \lrta H^1(\Q_v,E[p])$, we also have the following exact sequence
\[ 0 \lrta \Sel_p(E/\Q) \lrta H^1(\Q,E[p]) \lrta \prod_{v}H^1(\Q_v,E[p])/{\rm Im}(\kappa_v), \]
where  \[ \kappa_v : E(\Q_v)/ pE(\Q_v) \lrta H^1(\Q_v, E[p])\]
is associated to the multiplication by $p$ map  (called the {\it Kummer map  } of $E$) on the Galois cohomology groups 
\[ 0 \lrta E[p] \lrta E \stackrel{p}{\lrta} E  \lrta 0. \] 
Let 
 \[ \kappa_v^{ur} : E(\Q_v^{ur})/ pE(\Q_v^{ur}) \lrta H^1(\Q_v^{ur}, E[p])\]
 be the Kummer map of $E$ over the maximal unramified extension $\Q_v^{ur}$ of $\QQ_v$.
 
For every prime $v$ of $\QQ$, we get  restriction maps 
\[ \Sel_p(E/\Q) \stackrel{\ras_v}{\lrta} \Img (\kappa_v)\]
and 
\[ \Sel_p(E/\Q) \stackrel{\ras_v^{ur}}{\lrta} \Img(\kappa_v^{ur})\]
 %For a prime $v$ of $\QQ$, let $G_v$ denote the decomposition subgroup of $\gal(\overline{\QQ}/\QQ)$ and $I_v$ denote the inertia subgroup. 
For every prime $v\neq p$  of $\Q$, let $c_v(E)$ denote the Tamagawa number of $E$ over $\Q$. Under the assumption that $c_v$ is $p$-adic unit for every $v\neq p$, we shall show that $\ras_v^{ur}$ is the zero map. In particular, this implies that elements of $\Sel_p(E/\Q)$ are unramified outside $p$. Further, if $\ZZ/p\ZZ$-rank of $\Sel_p(E/\Q)$ is at least two then we shall
show that the kernel of $\ras_p^{ur}$ is non-trivial. Thus we get elements in $\Sel_p(E/\Q)$ which is unramified everywhere allowing us to construct  quotients of $\Cl_K/p\Cl_K$ isomorphic to $E[p]$ as Gal$(K/\Q) = \GL_2(\Z/p)$-modules.

Suppose that $E$ has good ordinary reduction at $p$. 
Then $\tilde{E}_p[p]$ is a rank one  $\mathbb{F}_p$-vector space, and is the maximal unramified quotient of $E[p]$ considered as a
$G_p = \Gal(\bar{\Q}_p/\Q_p)$-module.
Put, $a_p(E):= p+1-\tilde{E}_p(\mathbb{F}_p)$.  Since, $E$ has good ordinary reduction at $p$, $a_p(E)$ is $p$-adic unit and therefore  the polynomial $X^2-a_p(E)X+p$ has a unique $p$-adic unit root. Let  $\alpha_p$ be such a root.   Then, we have  $a_p(E)\equiv \alpha_p$ mod $p$. The Frobenius at $p$ acts  on $\tilde{E}_p[p]$ via multiplication
by $\alpha_p\equiv a_p(E)$ mod $p$. Let $C_p$ be the kernel of the natural quotient map $E[p] \lrta \tilde{E}_p[p]$. Then, $G_p$ acts on $C_p$ via the character $\omega_p\psi^{-1}$ where $\omega_p$ is the Teichmuller character at $p$.

Put $K=\Q(E[p])$.  Let \[\bar\rho : \gal(\overline{\QQ}/\QQ) \lrta \GL_2(\mathbb{F}_p) \] be the Galois representation given by the action of $\gal(K/\QQ)$ on $E[p]$. If $E$ has good ordinary reduction at $p$, the restriction of $\bar\rho$ to $G_p$ is equivalent to  the form 
\begin{equation}\label{rep}
  \left( {\begin{array}{cc}
   \omega_p\psi^{-1} & \star \\
   0 & \psi \\
  \end{array} } \right).
\end{equation}

Now suppose that $E[p]$ is an irreducible $\gal(\overline{\Q}/\Q)$-module. We shall need the following well-known lemma whose
proof shall be omitted.

\begin{lemma}\label{vanish}
  Let $G$ be a group operating on a  vector space $V$ over a field $F$.
  If $G$ has a central element which acts on $V$ by multiplication by an element of $F$ not equal to $1$,
  then  $H^i(G,V)=0$ for all $i\geq 0$.
\end{lemma}

As a consequence of the above Lemma we have the following (see \cite{LW}).

\begin{lemma}\label{vanish1}
Let $G$ be a finite subgroup of $\GL_2(\F_p)$ operating irreducibly on  $E[p]= \F_p + \F_p$,  then $H^i(G,E[p]) =0$ for all $i$.
\end{lemma}
\begin{proof}
  First suppose that  $p\nmid |G|$.  Then $H^i(G,E[p])=0$ for  all $i$ as $|G|$ annihilates $H^i(G,E[p])$.
  Now suppose that $G$ contains a non-trivial $p$-Sylow subgroup.  If a $p$-Sylow subgroup of $G$ is normal in $G$,
  then, up to  conjugacy, $G$ is contained in  the group of upper triangular matrices. 
  In particular, the action of $G$ on $E[p]$  is reducible, contrary to our assumption.
  Therefore a $p$-Sylow subgroup cannot be  normal in $G$, so there are at least 2
  distinct $p$-Sylow subgroups in $G$ which can be assumed to be the group of upper and lower triangular unipotent matrices in
  $\SL_2(\mathbb{F}_p)$. Since $\SL_2(\mathbb{F}_p)$ is generated by the group of upper and lower triangular unipotent matrices
    we get that $\SL_2(\mathbb{F}_p)$ is contained in $G$. In particular, $-I$ is also contained in $G$. Therefore the action of  $Z(G)$ on $E[p]$ is non-trivial. Lemma \ref{vanish} therefore  proves this lemma.  
\end{proof}

As a consequence of Lemma 1 if $E[p]$ is an irreducible representation of
$\Gal(\Q(E[p])/\Q)$, the restriction map
\[ H^1(\Q,E[p])\stackrel{\ras_K}{\lrta} H^1(K,E[p])\] 
is injective. Let $f\in H^1(\Q,E[p])$ be a non-trivial element. Then $\ras_K(f) : \gal(\overline{K}/K) \lrta E[p]$ is  a non-trivial homomorphism. Let $K_f$ be the field defined by the kernel of  $\ras_K(f)$. 
%Let $H_K$ be the maximal unramified abelian  extension of $K$. We have the following 
\begin{theorem}\label{main}
Suppose that the following holds:
\begin{enumerate}
\item[(a)] $E$ has good reduction at $p$. 
\item[(b)]  If $E$ has ordinary reduction at $p$,  $a_p(E) = 1$ mod $p$ and $E$ has no CM then  $\bar\rho$ is wildly ramified at $p$. 
\item[(c)] $c_v(E)$ is a $p$-adic unit for every finite prime $v\neq p$. 
\item[(d)] $E[p]$ is an irreducible $\gal(\overline{\Q}/\Q)$-module. 
\end{enumerate}
Then, $\rank_{\mathbb{F}_p}(\Ker(\ras_p^{ur}))\geq \rank_{\mathbb{F}_p}(\Sel_p(E/\Q)) -1$. Furthermore,  $\ras_K$ induces an injective homomorphism  \[ \ras_K : \Ker(\ras_p^{ur}) \lrta \Hom_{\gal(K/\Q)}(\Cl_K,E[p]) \subset H^1(K,E[p]).\]
%In particular, if ${\Sh}(E/\Q)$ is finite and  ${\Sh}(E/\Q)[p]\neq 0$ then $Cl_K[p]\neq 0$, where $Cl_K[p]$ denote the $p$-torsion subgroup of the ideal class group of $K$.
\end{theorem}
\begin{proof}
Let $f\in \Sel_p(E/\Q)$. First we shall show that $f$ is unramified outside $p$. Let $v\neq p$ be a finite prime of $\Q$. Consider the commutative diagram 
 {\small
\begin{equation}
\xymatrix{
0 \ar[r] & E(\Q_v)/pE(\Q_v) \ar[d]^{\lambda} \ar[r]^{\kappa_v} & H^1(\Q_v,E[p])  \ar[d]^{\mu} \\ 
0 \ar[r] & E(\Q_v^{ur})/pE(\Q_v^{ur}) \ar[r]^{\kappa_v^{ur}} &  H^1(\Q_v^{ur},E[p]) .    
}
\end{equation}
}
 To show that $f$ is unramified at $v$ it is enough to show that  $f$ is $p$ divisible in $E(\Q_v^{ur})$. This follows from the assumption that  $v$ is coprime to $p$ and $c_v(E)$ is a $p$-adic unit;
 see for example,
 \cite[Lemma 2]{BLR} and  \cite[Lemma 3.4, Lemma 3.6]{AS}. (Basically, this result follows from the fact that   $E(\Q_v^{ur})$
 can be treated as an extension of $\Z_v^{ur}$ by  one of the groups $\G_a(\bar{\F}_v), \G_m(\bar{\F}_v), E(\bar{\F}_v)$ depending on the reduction of $E$ mod $v$, and a finite group of order $c_v(E)$. Observe too, that this crucial fact
 $E(\Q_v^{ur})/pE(\Q_v^{ur})=0$ say at places of good reduction uses projectivity of the group variety $E$, and in particular,
 it does not hold for $\G_m$.) 

 If $v$ is an infinite place  then since $p$ is odd we again get that $f$ is unramified. 

If $f\in \Ker(\ras_p^{ur})$, then $f$ is unramified at $p$ and therefore it is unramified at all primes of $\Q$.
In this case, $\ras_K(f)$ is a Gal$(K/\Q)$-equivariant
homomorphism from  $\gal(\bar{K}/K) \lrta E[p]$ 
which is unramified at all prime of $K$ and which by Corollary \ref{vanish1}
induces an injective homomorphism from 
\[ \Ker(\ras_p^{ur}) \lrta \Hom_{\gal(K/\Q)}(\Cl_K,E[p]).\]

Next we show that $\Img(\ras_p^{ur})$ has $\mathbb{F}_p$-rank at most $1$, and hence
$\rank_{\mathbb{F}_p}(\Ker(\ras_p^{ur}))\geq \rank_{\mathbb{F}_p}(\Sel_p(E/\Q)) -1$ which will prove the theorem.

From the structure theory of rational points of an elliptic curve over a local field,
we have that $E(\Q_p)\cong \ZZ_p \oplus E(\Q_p)(torsion)$. If $a_p(E)\neq 1$ mod $p$ or $\rho$ is wildely ramified then we shall show that $E(\Q_p)[p]=0$. This implies that $\Img(\kappa_p)$ has has $\mathbb{F}_p$-rank  $1$. Since $\Img(\ras_p)\subset \Img(\kappa_p)$ and $\ras_p^{ur}$ factors through $\ras_p$,
proving that the $\Img(\ras_p^{ur})$ has $\mathbb{F}_p$ rank at most $1$.

First we consider the case when $E$ has good ordinary reduction at $p$ and $a_p(E) \neq 1$ mod $p$. 
From the assumption  that  $a_p(E) \neq 1$ mod $p$,  we have  $H^0(\Q_p,\tilde{E}_p[p])=0$. On the other hand since $\omega_p$ is ramified at $p$ and $\psi$ is unramified, $G_p$ acts on $C_p$ non-trivially. Thus $H^0(\Q_p,C_p)=0$. This shows that $H^0(\Q_p, E[p])=E(\Q_p)[p]=0$.  
If  $\star$ in equation (1) above is non-trivial, then since $\omega_p$ is non-trivial, we again see
that $H^0(\Q_p, E[p])=E(\Q_p)[p]=0$. 

In the supersingular  case, it is well-known 
that $E(\Q_p)[p]=0$ (see for example, \cite{E}, \cite[Proposition 8.7]{K}).
This again implies that $\rank_{\mathbb{F}_p}(\Img{\kappa_p^{ur})}=1$. 

Finally we consider the case when  $E$ has CM and $a_p(E)=1$ mod $p$.  This implies that in representation \eqref{rep}, $\psi$ is the trivial character.  For elliptic curves with CM we donot need to assume that $\rho$ is wildly ramified, even if $a_p(E)=1$
mod $p$ in which case we have $H^0(\QQ_p,\tilde{E}_p[p^\infty]) \cong \ZZ/p\ZZ$.

Since $E$ has good ordinary reduction at $p$, we have an  exact sequence of $G_p$-modules (see \cite[Section 2]{Gr}),
\begin{equation}\label{seq}
 0 \lrta E^+[p^\infty] \lrta  E[p^\infty] \lrta \tilde{E}_p[p^\infty] \lrta 0.  
 \end{equation}
Both groups,  $\tilde{E}_p[p^\infty] $ and $E^+[p^\infty] $ are isomorphic to $\QQ_p/\ZZ_p$.   The $G_p$-module $\tilde{E}_p[p^\infty] $ is unramified and the inertia group $I_p$ at $p$ acts on $E^+[p^\infty]$ via the $p$-th cyclotomic character $\chi_p$. Further, since $E$ has CM, the corresponding dual exact sequence is  split exact sequence when tensored with $\Q$, from which we deduce that
$H^0(\QQ_p^{ur},\tilde{E}[p^\infty])\cong \QQ_p/\ZZ_p \subset H^0(\QQ_p^{ur},E[p^\infty])$.  From the structure theorem
of submodules of $\Q_p/\Z_p+\Q_p/\Z_p$ and using the fact that $\Z/p\Z + \Z/p\Z$ cannot be contained in $H^0(\QQ_p^{ur},E[p^\infty])$,
it follows that $ H^0(\QQ_p^{ur},E[p^\infty]) \cong \Q_p/\Z_p$.
Therefore, a non-zero element in $H^0(\QQ_p,E[p^\infty])\cong \ZZ/p\ZZ$ is not $p$-divisible, but it is $p$-divisible in $H^0(\QQ_p^{ur},E[p^\infty])$. Thus we get an element in $E(\QQ_p)$ which is non $p$-divisible, but $p$-divisible in $E(\QQ_p^{ur})$. In particular, the natural map
\begin{equation}\label{base}
\phi : E(\QQ_p)/pE(\QQ_p) \lrta E(\QQ_p^{ur})/pE(\QQ_p^{ur})  
 \end{equation}
 is not injective. Since $H^0(\QQ_p,E[p^\infty])\cong \ZZ/p\ZZ$, we get that $E(\QQ_p)\cong \ZZ_p\oplus \ZZ/p\ZZ$. Therefore, $\mathbb{F}_p$-rank of $E(\QQ_p)/pE(\QQ_p)$ is two. Since $\ras_p^{ur}$ factors through $\phi$ and $\kappa_p^{ur}$ and $\phi$ is not injective, we get that the $\Img(\ras_p^{ur})$ has $\mathbb{F}_p$-rank $\leq 1$. Hence, the claim follows.
\end{proof}

\begin{cor}
  Under the hypothesis of the theorem on the elliptic curve $E$ over $\Q$, if either
  rank$_{\F_p}(\Sh(E) [p]) > 1 $, or  $\Sh(E) [p] \not = 0 $, and $\Sh(E)[p^{\infty}] < \infty$, 
  then there exists an
  unramified abelian extension of $\Q(E[p])$ with Galois group $E[p]$, Galois over $\Q$,  on which
 $\Gal(\Q(E[p])/\Q) \subset \GL_2(\Z/p)$ operates by the standard
  2 dimensional representation  of $ \GL_2(\Z/p)$ over $\Z/p$. If $\Sh(E) [p] = 0$ but then if Mordell-Weil rank of $E$ over $\Q$ is $\geq 2$,
  then also there exists such an
  unramified abelian extension of $\Q(E[p])$ with Galois group $E[p]$.
\end{cor}
\begin{proof} Observe that if $\Sh(E)[p^{\infty}] < \infty$, then $\Sh(E) [p]$ has order which is a square, in particular,
  if $\Sh(E) [p]$ is nonzero, its rank over $\Z/p$ is at least 2. Note also that under the hypothesis of the
  Theorem \ref{main},
  the action of  $\Gal(\Q(E[p])/\Q) \subset \GL_2(\Z/p)$ on $E[p]$ is irreducible, and hence  any non-trivial homomorphism
  $\ras_K(f) : \gal(\overline{K}/K) \lrta E[p]$ which is $\Gal(K/\Q)$-equivariant must be surjective, constructing an
  unramified extension of $K$ with Galois group $E[p]$.
\end{proof}

\begin{remark} In Theorem \ref{main}, all our analysis is done with the Selmer group, not distiguishing the part of it
  coming from the Mordell-Weil group, and the part coming from $\Sh$. In fact there have been papers constructing unramified
  extensions  of $\Q(E[p])$ using  the Mordell-Weil group,
  see e.g. \cite{SY} and \cite{TH}, which give conclusions of the
 form that if the Mordell-Weil group is sufficiently large compared to number of places of $\Q$ where $E$ has bad reduction,
then the classgroup     of $\Q(E[p])$ is nonzero. It may be remarked that dealing with extensions  of $\Q(E[p])$ using  the Mordell-Weil group is easier in the sense that these extensions come from extensions of $\Q$.
\end{remark}

\noindent
{\bf Example 1} Let $E$ be the elliptic curve defined by the equation 
\[  y^2 + x y = x^{3} -  x^{2} - 332311 x - 73733731  .\]
The Cremona level of $E$ is $1058d1$. Then it follows from \cite{LMFDB} that $E[5]$ is an irreducible representation with full image of Galois representation equal to $\GL_2(\mathbb{F}_5)$. Further the Tamagawa number is equal to $1$ at primes of bad reduction $2$ and $23$. Again from \cite{LMFDB}, we have that ${\Sh}(E/\Q)[5]$ has $\mathbb{F}_5$-rank two and $a_p(E)=-2\neq 1$ mod $5$. As a consequence we get that the ideal class group $\Cl_K$ for $K=\Q(E[5])$ has a quotient isomorphic to $E[5]$ as $\gal(K/\Q)$-module. 

We mention that to produce a quotient of $\Cl_K/p\Cl_K$ isomorphic to $E[p]$, it is not necessary to assume that ${\Sh}(E/\Q)[p]$ is non-trivial. The elliptic curve, 
\[ F :  y^2 + x y + y = x^{3} + 2  \] 
given by Cremona number $1058c1$ satisfy assumptions (a) to (d) of Theorem \ref{main}. In fact, $F[5]\cong E[5]$ as $\gal(\bar{\QQ}/\QQ)$-module (see \cite[Table, page 25]{CM}). The Mordell-Weil rank of $F$ is two and ${\Sh}(F/\QQ)[5]=0$. 

\section{The converse}

In this section, we address the converse part of Question \ref{ques}.  
 
 Let $R_p(E/\QQ)$ be the subgroup of $H^1(\QQ,E[p])$ defined by the exact sequence
 \[ 0 \lrta R_p(E/\QQ) \lrta H^1(\QQ,E[p]) \lrta \prod_v H^1(\QQ_v^{ur},E[p])\]
 
 \begin{lemma}\label{fin}
Suppose that the following holds
\begin{enumerate}
\item[(a)] $E$ has good  reduction at $p$. 
\item[(b)]  If $E$ has ordinary reduction at $p$ and  $a_p(E) \neq 1$ mod $p$ then $E[p]$ is wildly ramified at $p$. 
\item[(c)] $c_v(E)$ is a $p$-adic unit for every finite prime $v\neq p$.
\end{enumerate}
Then $R_p(E/\QQ)\subset \Sel_p(E/\QQ)$. 
 \end{lemma}
 \begin{proof}
 Let $f\in R_p(E/\QQ)$.  Then $f|_{I_v}$ is zero in $H^1(\Q_v^{ur},E[p])$ and therefore the image of $f$ in $H^1(\Q_v^{ur},E)$ is also zero.  We suppose first that $v\neq p$.   Since $c_v(E)$ is a $p$-adic unit, $H^1(\Q_v,E)\lrta H^1(\Q_v^{ur},E)$ is injective and therefore the image of $f$ in $H^1(\Q_v,E)$ is zero.   Thus to show that $f\in \Sel_p(E/\Q)$, it is enough to show that the image of $f$ in $ H^1(\Q_p,E)$ is also zero, which is what we do next. 
 
 Suppose that $E$ has good and ordinary reduction at $p$. If $\star$ in representation \ref{rep} is non-zero then $E[p]^{I_p}=0$ and we get that $H^1(\Q_p^{ur}/\Q_p,E[p]^{I_p})=0$. If $\star$ is zero, from the assumption $a_p(E)\neq 1$ mod $p$, the Frobenius at $p$ acts on $E[p]^{I_p}$ by a non-trivial character of order prime to $p$. This implies that $H^0(\Q_p^{ur}/\Q_p,E[p]^{I_p})=0$. Since $\gal(\Q_p^{ur}/\Q_p)$ is procyclic, we get that $H^1(\Q_p^{ur}/\Q_p,E[p]^{I_p})=0$. Therefore $H^1(\Q_p,E[p]) \lrta H^1(\Q_p^{ur},E[p])$ is injective. This implies that $f$ is trivial in $H^1(\Q_p,E[p])$. Therefore the image of $f$ in $H^1(\Q_p,E)$ is zero. 
 
Next, consider the case when $E$ has supersingular reduction at $p$. Then,  as is well-known,  $E(F)[p]=0$ for every finite unramified extension $F$ of $\QQ_p$ (see for example \cite{E}). This implies that $E(\QQ_p^{ur})[p]=0$. Therefore the restriction map from $H^1(\QQ_p,E[p])\lrta H^1(\QQ_p^{ur},E[p])$ is injective. 
Again we have that  the image of $f$ in  $H^1(\QQ_p,E[p])$ trivial. In particular, its image in $H^1(\QQ_p,E))$ is also trivial. This proves  that $f\in \Sel_p(E/\QQ)$. 
 \end{proof}
 
 For the next lemma, we will keep the notation $K=\QQ(E[p])$,  $I_v$ the inertia subgroup of
 $\gal(\overline\QQ_v/\QQ_v)$ for a prime $v$ of $\QQ$, and for a prime $w|v$ of $K$,  $I_w^{K}$ will
 denote the corresponding inertia subgroup of $\gal(\overline{K}_w/K_w)$.

\begin{lemma}\label{nonzero}
  Let $\mathfrak{T}$ be the set of finite primes $v$ of $\Q$ not equal to $p$ such that $E$ has
  mutiplicative reduction at $v$ and $E(\Q_v^{ur})[p]$ has $\mathbb{F}_p$-rank one. In the case $p=3$,
  we also assume that $\mathfrak{T}$ contains the set of prime $v$ such that $E$ has additive
  potentially good reduction with $E(\Q_v^{ur})[p]$ has $\mathbb{F}_p$-rank one.
  Let $v\neq p$ be a finite prime of $\Q$ for which $c_v(E)$ is a $p$-adic unit.
  Then $H^1(I_v/I_w^K,E[p])$ is non-trivial for a prime $w|v$ of $K$ if and only if $v\in \mathfrak{T}$.
  Further, in this case the $\mathbb{F}_p$-rank of $H^1(I_v/I_w^K,E[p])$ is also one.
\end{lemma}
\begin{proof}
  First suppose that $E$ has additive reduction at $v \not = p$. If $E$ has potentially multiplicative reduction at $v$,
  there exist a ramified quadratic extension  of $\Q_v$ over which $E$ has split multiplicative reduction.  Let $\chi$ be the corresponding ramified quadratic character of the absolute Galois group of $\Q_v$. This implies
  that $E$ is a ramified quadratic twist of an elliptic curve $E'$ with split multiplicative reduction at $v$ and $E[p]\cong E'[p]\otimes \chi$ as $\gal(\overline{\Q_v}/\Q_v)$-modules.
  Recall that $\bar\rho$ denote the Galois representation associated to $E[p]$ and let $\bar{\rho}'$ denote the Galois representation associated to $E'[p]$. Since $E'$ has split multiplicative reduction, $E'[p]$ is a semistable Galois representation. Therefore $\bar\rho'$ restricted to $I_v$ is represented
  by matrices of the form   $\begin{pmatrix} 
1 & \star \\
0 & 1 
\end{pmatrix}
$.
This implies that $\bar\rho$ is  represented by matrices of the form  $\begin{pmatrix} 
\chi & \star \\
0 & \chi 
\end{pmatrix}
$. Since $\chi $ is a ramified character, we get that $E[p]^{I_v}=0$.

Next suppose that $E$ has additive reduction at $v \not = p$ which is potentially good. In this case it is well-known that  $I_v/I_w^K$ is a cyclic group of order at most $6$. 
Thus if $p\geq 5$ then  $I_v/I_w^K$, has order prime to $p$ and therefore $H^1(I_v/I_w^K,E[p])=0$. Finally,
suppose that   $p=3$  and $E$  has additive and potentially good reduction at $v$. In this case $I_v$ acts on
$E[p]$ non-trivially  (see for example \cite[Corollaries 2(b), 3]{ST}).  
Since $I_v/I_w^K$ is cyclic, the cardinality of $H^1(I_v/I_w^K,E[p])$ and
$H^0(I_v/I_w^K,E[p])=E(\Q_v^{ur})[p]$ are the same. As $I_v$ acts non-trivially on $E[p]$,
the $\mathbb{F}_p$-rank of $E[p]^{I_v}$ is at most one. Therefore if $E[p]^{I_v}$ is non-zero,
the $\mathbb{F}_v$-rank of $H^1(I_v/I_w^K,E[p])$ is  one.  We mention that it is well-known that
$E(\Q_v^{ur})[p]\neq 0$ only when $3$ divides the Tamagawa number of $E$ over $\Q_v(\mu_3)$. 

Next, we suppose that $E$ has multiplicative reduction at $v \not = p$. If $E$ has non-split multiplicative reduction at $v$, then there exists a quadratic unramified  extension of $\Q_v$ over which $E$ has split multiplicative reduction. Thus if $E$ has multiplicative reduction at $v$ then it has split multiplicative reduction over $\Q_v^{ur}$. This implies that $\mu_p\subset E(\Q_v^{ur})[p]$. Thus $E[p]^{I_v}\neq 0$. If $E[p]^{I_v}$ has $\mathbb{F}_p$-rank two then $I_v/I_w^K$ is a trivial group. Thus $H^1(I_v/I_w^K,E[p])$  can be non trivial only when $E(\Q_v^{ur})[p]$ has $\mathbb{F}_p$-rank one. 

Now suppose that $E(\Q_v^{ur})[p]=E[p]^{I_v}$ has $\mathbb{F}_p$-rank one and $E$ has multiplicative reduction at $v$. Since $\mu_p\subset \Q_v^{ur}$, and the determinant of the Galois representation $\bar\rho$ associated to $E[p]$ is the cyclotomic character mod $p$, we get that $\bar\rho$ restricted to $I_v$ is equivalent to the matrices of the form   $\begin{pmatrix} 
1 & \star \\
0 & 1 
\end{pmatrix}
$ for $\star\in \mathbb{F}_p$. Further, since $E[p]^{I_v}\neq E[p]$, $I_v$ acts on $E[p]$-nontrivially. Therefore the image of $\bar\rho$ has order $p$ and it is a $p$-Sylow subgroup of $G=\gal(K/\Q)$. In particular, $I_v/I_w^K$ is a cyclic  group of order $p$. This implies that the $\mathbb{F}_p$-rank of $H^1(I_v/I_w^K,E[p])$ is one. This proves the lemma. 
\end{proof}

\begin{theorem}\label{mainconv}
 Suppose that the following holds
\begin{enumerate}
\item[(a)] $E$ has good  reduction at $p$. 
\item[(b)] If $E$ has ordinary reduction at $p$ and  $a_p(E) \neq 1$ mod $p$ then  that $E[p]$ is wildly ramified at $p$. 
\item[(c)] $c_v(E)$ is a $p$-adic unit for every finite prime $v\neq p$.
\item[(d)] $E[p]$ is an irreducible $\gal(\bar{\QQ}/\QQ)$-representation. 
\end{enumerate}
Then \[ \mathrm{rank}_{\mathbb{F}_p} \Hom_G(\Cl_K/p\Cl_K,E[p])  \leq \mathrm{rank}_\mathbb{\mathbb{F}_p} \Sel_p(E/\Q) +\# \mathfrak{T}.\] 
%If $\Cl_K/p\Cl_K$ has a quotient isomorphic to $E[p]$  and the Mordell-Weil rank of $E$ over $\QQ$ is zero then  ${\Sh}(E/\Q)[p]\neq 0$.
 \end{theorem}
 \begin{proof}
 Consider the commutative diagram 
  {\small
\begin{equation*}
\xymatrix{
0 \ar[r] & \Hom(\Cl_K/p\Cl_K,E[p])^G   \ar[r] &  H^1(K,E[p])^G     \ar[r] & \prod_{v} (\prod_{w|v} H^1(K_w^{ur} ,E[p]) )^{G}     \\ 
0 \ar[r] & R_p(E/\Q) \ar[u]^{\alpha} \ar[r] & H^1(\Q,E[p])  \ar[u]^{\beta}  \ar[r]  & \prod_v H^1(\Q_v^{ur},E[p])  \ar[u]^{\prod\gamma_v}
}
\end{equation*}
}
where $v$ varies over the primes of $\Q$ and $w$ varies over primes of $K$. From Lemma  \ref{vanish1},  $\beta$ is an isomorphism. 

The kernel of $\gamma_v : H^1(\Q_v^{ur},E[p]) \lrta  \prod_{w|v} H^1(K_w^{ur} ,E[p]) )^{G} $ is same as the kernel of $H^1(\Q_p^{ur},E[p]) \lrta   H^1(K_w^{ur} ,E[p]) )$ for fixed prime $w|v$ of $K$, which is equal to $H^1(I_v/I_w^K,E[p])$. 

Let $w|p$ be a prime of $K$. We shall show that $H^1(I_p/I_w^K,E[p])=0$. First we consider the case when $E$ has supersingular reduction at $p$. It is well-known  
that the action of $I_p$ on $E[p]$ factors through a cyclic group  of order prime to $p$, see for example \cite{E},
and also \cite[Theorem 1.2]{BG} for an exposition. As a consequence, we get that $H^1(I_p/I_w^K,E[p])=0$.

Next, consider the case when $E$ has ordinary reduction at $p$. If $\star=0$ in the representation \ref{rep} then $I_p/I_w^K$ has order prime to $p$ and therefore $H^1(I_p/I_w^K,E[p])=0$. Now suppose that $\star\neq 0$. Let $P$ be the unique $p$-sylow subgroup of $H:=I_p/I_w^{K}$. Then it is enough to show that $H^1(P,E[p])^{H/P}=0$. We have an exact sequence of $H$-modules
\[ 0 \lrta A \lrta  E[p] \lrta B \lrta 0 \]
where $P$ acts trivially on   $\mathbb{F}_p$-rank one modules $A$ and $B$.
By considering the corresponding long exact sequence of cohomology groups, it is easy to see that
the natural  map $H^1(P,E[p])\rightarrow H^1(P,B)$ induced by the surjective map $E[p] \rightarrow B$
is in fact an isomorphism. 

 Since $H/P$ acts trivially on $B$ and non trivially on $P$, we get that $H^1(P,B)^{H/P}=0$. Thus $H^1(P,E[p])^{H/P}=0$ and therefore $H^1(H,E[p])=H^1(I_p/I_w^K,E[p])=0$.

 As a consequence of the snake lemma,  the above commutative diagram and Lemma \ref{nonzero},  we have
\[ \mathrm{rank}_{\mathbb{F}_p} \Hom_G(\Cl_K/p\Cl_K,E[p])  \leq \mathrm{rank}_{\mathbb{F}_p} R_p(E/\Q) +\# \mathfrak{T}. \]
We mention that if $\mathfrak{T}$ is an empty set then in fact we have $\mathrm{rank}_{\mathbb{F}_p} \Hom_G(\Cl_K/p\Cl_K,E[p])  = \mathrm{rank}_{\mathbb{F}_p} R_p(E/\Q) $.
From Lemma \ref{fin}, $R_p(E/\Q)\subset \Sel(E/\Q)$ and this proves the theorem. 
 %Further, if $E$ has split multiplicative reduction at $v$ then the assumption that $c_v$ is a $p$-adic unit implies that $E(\Q_v^{ur})[p]$ is isomorphic to the group of $p$-th root of unity and therefore $I_v$ acts on $E[p]$ nontrivially. 
\end{proof}
\noindent
{\bf Example 2 : } Consider the following  elliptic curve with Cremona number  $423801ci1$ and conductor $3^{2} \cdot 7^{2} \cdot 31^{2}$ from \cite{LMFDB}. 
$$ E :  y^2 + y = x^{3} - 17034726259173 x - 27061436852750306309   $$
From \cite{LMFDB}, we get that $E[5]$ is an irreducible $\gal(\overline{\Q}/\Q)$-module with the image of associated Galois representation as $\GL_2(\mathbb{F}_5)$.  The curve $E$ has additive reduction at the primes $3$, $7$ and $31$  and has good and ordinary reduction at $5$ with $a_5(E)=4\neq 1$ mod $5$. The Tamagawa numbers are $5$-adic unit at all primes of bad reduction. In this case the set $\mathfrak{T}$ is empty. The Mordell-Weil rank of $E$ is $0$ and $\Sh(E/\Q)$ has cardinality  $625$. We do not know if $\Sh(E/\Q) \cong \Z/25\Z\oplus \Z/25\Z $ or $\Sh(E/\Q) \cong \Z/5\Z\oplus \Z/5\Z \oplus \Z/5\Z \oplus \Z/5\Z $. In any case, as a consequence of  Theorem \ref{main}, we get that $ \mathrm{rank}_{\mathbb{F}_5} \Hom_G(\Cl_K/5\Cl_K,E[5])$ is at least one. Further, since $\mathfrak{T}$ is empty, from  Theorem \ref{mainconv} we get that $R_5(E/\Q)\cong   \Hom_G(\Cl_K/5\Cl_K,E[5])$ and it  is a non-trivial group.   We also mention that being a subgroup of $\Sel_p(E/Q)$, in general  $R_5(E/\Q)$  may be the trivial group even if $\Sel_p(E/\Q)$ is non-trivial. But in this particular example, $R_5(E/\Q)$ is a non trivial group and every element of $ \Hom_G(\Cl_K/5\Cl_K,E[5])$ is in the image of $R_5(E/\Q)$ under the natural restriction map of Galois cohomology.

\vspace{5mm}

\noindent{\bf Acknowledgement:} The authors  must thank Harish-Chandra Research Institute, Allahabad for the hospitality, and
Prof. Kalyan Chakraborty for the conference on Class groups 
at the  Harish-Chandra Research Institute in October 2019, Allahabad where our collaboration on this work took place.
The authors also thank N. Dummigan, H. Katsurada, and  C. Wuthrich for their interest and encouragement.

\end{document}